\documentclass{article}
\usepackage[utf8]{inputenc}
\usepackage{graphicx} 
\usepackage{amsmath,amsfonts,amssymb,amsthm}
\usepackage[mathscr]{eucal}
\usepackage{amscd}
\usepackage{tikz}
\usepackage{tkz-graph}
\usepackage{multicol}
\usepackage{float}
\usepackage{authblk}

\newtheorem{theorem}{THEOREM}[section]
\newtheorem{lemma}[theorem]{LEMMA}

\newtheorem{remark}[theorem]{REMARK}

\newcommand*{\legendre}[2]{\genfrac(){}{}{#1}{#2}}

\title{On the Fractional Parts of Polynomials Modulo $p$}
\author{Xuejun Guo$^1$, Chen Lin$^2$\footnote{Corresponding author.}, Zhefeng Xu$^3$}
\affil{{\small {$^{1,2}$School of Mathematics, Nanjing University, Nanjing 210093, China}\\{$^3$Research Center for Number Theory and Its Applications, Northwest University, Xi'an 710127, China}\\
$^1$guoxj@nju.edu.cn, 
$^2$chen.lin@smail.nju.edu.cn, $^3$zfxu@nwu.edu.cn} }
\date{}
\date{}

\begin{document}

\maketitle
\begin{abstract}
We study a half-interval distribution problem for polynomial residues modulo an odd prime $p$: how often the fractional part of $\varphi(x)/p$ lies in the upper half of the unit interval as $x$ ranges over $1\leq x< p/2$. Using finite Fourier expansions together with the Weil bound, we prove an asymptotic formula $\#\left\{1\leq x< p/2:\left\{{\varphi(x)}/{p}\right\}>\frac12\right\}
=\frac{p}{4}+O_\varphi(\sqrt p\log^2 p).
$
We then show that the error term can be improved to $O_\varphi(\sqrt p\log p)$ for arbitrary quadratic polynomials and for polynomials satisfying suitable reflection symmetries. For even monomials $\varphi(x)=x^m$, we further obtain the bound $O_m(\sqrt p\log\log p)$ under the Generalized Riemann Hypothesis. Finally, in the case $m=2$, we prove an unconditional matching lower bound, showing that the factor $\log\log p$ is best possible in this setting.
\end{abstract}
  \noindent { 2020\it Mathematics Subject Classification: 11L07, 11L40, 11T23} 
  \\[1mm]
    \noindent {\it Keywords: Fractional parts of polynomials, Incomplete exponential sums, Weil bound, Finite fields, Dirichlet characters}

\section{Introduction}
The purpose of this paper is to study a finite analogue of Weyl's equidistribution problem for polynomial values modulo a prime \cite{Weyl}, with the variable restricted to a half-interval.  Let \(p\) be an odd prime and let \(\varphi\in\mathbb Z[x]\).  We are interested in the distribution of the fractional parts
$
    \left\{{\varphi(k)}/p\right\}, 1\leq k< p/2.
$
If \(k\) ranges over a complete residue system modulo \(p\), the problem is governed by complete exponential sums and hence by the Weil bound.  The restriction \(1\leq k< p/2\), however, leads naturally to incomplete exponential sums, and this is the main source of difficulty.

Let $m>1$ be an integer and $p$ be an odd prime. Define the counting function 
$$f_m(p):=\#\left\{1\leq k<\frac{p}{2}: \left\{\frac{k^m}{p}\right\}>\frac{1}{2}\right\},$$
where $\{x\}$ denotes the fractional part of a real number $x$.
For the monomial \(\varphi(x)=x^m\), Zhi-Wei Sun \cite{SZW} conjectured that $f_m(p) = \frac{p}{4} + O_m(\sqrt{p})$ in 2018.  This is a natural square-root cancellation prediction, since the expected main term is one half of the half-interval.

One of our main conclusions is that Sun's predicted error term fails
already for \(m=2\).
We further  investigate a generalized version of this problem. For any polynomial $\varphi(x) \in \mathbb{Z}[x]$ of degree $m \ge 2$, we define
$$f_\varphi(p) := \#\left\{1 \le k < \frac{p}{2} : \left\{\frac{\varphi(k)}{p}\right\} > \frac{1}{2}\right\}.$$
We establish the estimate for $f_\varphi$ as follows. 
\begin{theorem}
\label{general}
   Let \(\varphi(x)\in\mathbb Z[x]\) be a fixed polynomial of degree \(m\ge 2\). For any odd prime $p$ for which the reduction of $\varphi$ modulo $p$ has degree at least 2, we have
    $$\left|f_\varphi(p)-\frac{(p-1)^2}{4p}\right| \leq \frac{4(m-1)}{\pi^2}\sqrt{p}\log(3p)\left(\frac{1}{2}+\frac{4}{\pi^2}\log(3p)\right).$$
    In particular,
    $$f_\varphi(p)=\frac{p}{4}+O_\varphi(\sqrt{p}\log^2p).$$
\end{theorem}

\begin{remark}
\label{remark1.2}
    (1) For odd prime $p$ for which the reduction of $\varphi$ modulo $p$ has degree 0 or 1, we use the trivial estimate 
    $$\left|f_\varphi(p)-\frac{(p-1)^2}{4p}\right| \leq\left|\frac{p-1}{2}-\frac{(p-1)^2}{4p}\right|=\frac{p^2-1}{4p}.$$
    Since $\varphi$ is fixed, its reduction modulo $p$ is constant or linear for only finitely many primes. These primes can therefore be absorbed into the implied constant in the $O_\varphi$-term.\\
    (2) The symbol $O_\varphi$ means that the implied constant depends on the degree and coefficients of $\varphi$.
\end{remark}
The $\log^2 p$ factor in Theorem \ref{general} is not necessarily optimal for polynomials of small degrees. In particular, for the quadratic case, we establish the following improved estimate.

\begin{theorem}
\label{quadratic}
    Let $\varphi(x)=Ax^2+Bx+C \in \mathbb{Z}[x]$ be a fixed quadratic polynomial with $A \neq 0$. For any odd prime $p\nmid A$,
    \[
        \left|f_\varphi(p)-\frac{(p-1)^2}{4p}\right|\leq\frac{4}{\pi^2}(7.0508\sqrt{p}+2)\log(3p).
    \]
    In particular,  
    $$f_\varphi(p)=\frac{p}{4}+O_\varphi(\sqrt{p}\log p).$$
\end{theorem}

For polynomials of higher degrees, if we further assume that the polynomial $\varphi(x)$ has reflection symmetry, the incomplete exponential sum can be expressed in terms of complete exponential sums, which allows us to refine the error term.
\begin{theorem}
\label{general_symmetry}
    Let $\varphi(x) \in \mathbb{Z}[x]$ be a given polynomial of degree $m \ge 2$. Suppose there exists a fixed integer $c$ such that $\varphi(c-x) = \varphi(x)$. For any odd prime $p$ for which the reduction of $\varphi$ modulo $p$ is nonconstant, we have 
    $$\left|f_\varphi(p)-\frac{(p-1)^2}{4p}\right|\leq \frac{2}{\pi^2}\left((m-1)\sqrt{p}+2|c|+3\right)\log(3p).$$
    In particular,
    $$f_\varphi(p) = \frac{p}{4} + O_\varphi(\sqrt{p}\log p).$$
\end{theorem}

\begin{remark}
    This covers the case where $\varphi(x) = x^m$ for an even integer $m$. 
\end{remark}

Returning to Sun's original question, we first focus on the quadratic case $m=2$. In this case, we can derive an exact formula for the counting function using Dirichlet's class number formula. This exact unconditional evaluation allows us to establish a sharp lower bound for the error term via the extreme values of Dirichlet $L$-functions.   
\begin{theorem}\label{m=2}
    For an odd prime $p>3$, we have 
    \[
f_2(p)=
\begin{cases}
    \frac{p-1}{4} & \text{if } p\equiv 1\pmod{4},\\
    \frac{p-1}{4}-\frac{3}{2}h(-p) & \text{if } p\equiv 3\pmod{8},\\
    \frac{p-1}{4}-\frac{1}{2}h(-p) & \text{if } p\equiv 7\pmod{8},
\end{cases}
\]
    where $h(-p)$ denotes the class number of the imaginary quadratic field $\mathbb{Q}(\sqrt{-p})$. 
    Furthermore, there exists an absolute constant $C>0$ such that, for infinitely many primes $p\equiv 3\pmod{4}$, we have 
    $$ \left|f_2(p) - \frac{p}{4}\right| \geq C\sqrt{p}\log\log p. $$
\end{theorem}

As an immediate consequence of this lower bound, the previously expected $O(\sqrt{p})$ error term cannot be achieved uniformly, which disproves Sun's conjecture.

While the exact evaluation is restricted to $m=2$, we can investigate higher even powers $m \ge 2$ by noticing that the monomial $x^m$ has a rich multiplicative structure over finite fields. This observation allows us to translate the counting problem into the estimation of incomplete Dirichlet character sums. Assuming the Generalized Riemann Hypothesis (GRH), we can generalize the asymptotic behavior observed in $f_2(p)$ to all even integers $m$.

\begin{theorem}[Conditional on GRH]
\label{estimate_GRH}
    Let $m \geq 2$ be an even integer and consider the monomial $\varphi(x) = x^m$. Assuming the Generalized Riemann Hypothesis, for any odd prime $p$, we have
    $$f_m(p) = \frac{p}{4} + O_m(\sqrt{p}\log\log p).$$
\end{theorem}

Finally, we consider a natural generalization of the counting function discussed above. The improved bound for even polynomials relies on pairing $x$ and $-x$, which corresponds to the group of square roots of unity \(\{\pm1\}\). A natural question is whether we can extend this idea to higher-order symmetries involving the $d$-th roots of unity. To answer this, we generalize the half-interval to a representative set $A$ (a fundamental domain) modulo $p$, and establish a similar uniform bound of the counting function for the polynomial $x^d$.

\begin{theorem}
\label{fundamental domain}
     Let $d \geq 2$ be a fixed integer. Let $p$ be a prime such that $p \equiv 1 \pmod d$, and let $\omega \in \mathbb{F}_p^*$ be a primitive $d$-th root of unity. Let $A \subset \{1, 2, \dots, p-1\}$ be a set of representatives forming a fundamental domain for the action of the group $G_d = \{1, \omega, \dots, \omega^{d-1}\}$ on $\mathbb{F}_p^*$, meaning $|A| = \frac{p-1}{d}$ and $\mathbb{F}_p^* \equiv \bigsqcup_{i=0}^{d-1} \omega^i A \pmod p$. Then the counting function $f_A(p) := \#\left\{x \in A : \left\{ \frac{x^d}{p} \right\} > \frac{1}{2} \right\}$ satisfies
     $$\left|f_A(p)-\frac{(p-1)^2}{2dp}\right|\leq\frac{4}{d\pi^2}\left((d-1)\sqrt{p}+1\right)\log(3p).$$
     In particular, 
     $$f_A(p)=\frac{p}{2d}+O_d(\sqrt{p}\log p).$$
\end{theorem}

The main idea of our proofs is to use finite Fourier expansion. We express the indicator functions of the required intervals as finite Fourier series, which directly translates the counting problem into the estimate of exponential sums. For general polynomials, we estimate these sums using the well-known Weil bound. For quadratic or symmetric cases, we improve the estimate by
using bounds for incomplete Gauss sums
 or symmetry properties respectively. For the monomial case, we use Dirichlet characters instead, which allows us to apply the Generalized Riemann Hypothesis to get a better error term.

The remainder of this paper is organized as follows. In Section \ref{preliminaries}, we list the necessary preliminary lemmas. Section \ref{proof of main results} provides the unconditional proofs for general, quadratic, and symmetric polynomials (Theorems \ref{general}, \ref{quadratic}, and \ref{general_symmetry}). Section \ref{original case} focuses on the monomial case, including the specific results for $m=2$ and the conditional bound under Generalized Riemann Hypothesis (Theorem \ref{m=2} \& Theorem \ref{estimate_GRH}). Finally, in Section \ref{section_generalizaion}, we prove the generalized result for the representative set $A$ (Theorem \ref{fundamental domain}).

\medskip
\noindent\textbf{Notation.} The following notations will be used throughout this paper:
\begin{itemize}
    \item $p$ always denotes an odd prime;
    \item $\{x\}$ denotes the fractional part of a real number $x$;
    \item $\bar{\varphi}$ denotes the reduction of $\varphi$ modulo $p$;
    \item $f = O(g)$ and $f \ll g$ mean $|f| \le cg$ for some unspecified positive constant $c$. If the implied constant $c$ depends on certain parameters, we indicate this dependence by adding subscripts, such as $O_\varphi(g)$, $\ll_\varphi$, $O_m(g)$, or $\ll_d$. In particular, the subscript $\varphi$ (as in $O_\varphi$ or $\ll_\varphi$) means that the implied constant depends on the degree and the coefficients of the polynomial $\varphi$;
    \item For a finite set $S$, $|S|$ (or $\# S$) denotes its cardinality, and $\mathbf{1}_S(x)$ denotes its indicator function, which equals $1$ if $x \in S$ and $0$ otherwise;
    \item $\legendre{\cdot}{\cdot}$ denotes the Legendre symbol;
    \item $\mathbb{F}_p$ denotes the finite field with $p$ elements, and $\mathbb{F}_p^*$ denotes the multiplicative group of all non-zero elements in $\mathbb{F}_p$;
    \item $e_p(x) := e^{2\pi i x/p}$ denotes the standard additive character.
\end{itemize}

\section{Preliminaries}
\label{preliminaries}
In this section, we collect the necessary tools for our proofs. We divide them into two parts: estimates for additive exponential sums, and results concerning Dirichlet characters and $L$-functions.

First, we state three lemmas about additive exponential sums, dealing with complete sums, finite Fourier expansions, and incomplete quadratic sums, respectively. The following well-known Weil bound is important for estimating the sum over a complete residue system.
\begin{lemma}[Weil bound, \cite{FF_LN}, Theorem 5.38]
\label{Weil bound}
    Let $f(x) \in \mathbb{F}_p[x]$ be a polynomial of degree $m \geq 1$. If $(p,m)\neq(2,2)$, then
    $$\left| \sum_{x=0}^{p-1} e_p(f(x)) \right| \le (m-1)\sqrt{p}.$$
\end{lemma}
When $p\nmid m$, this is the classical Weil bound cited above. If $p\mid m$ and $(p,m)\ne(2,2),$ then the result follows from the trivial estimate
$\left|\sum_{x\in\mathbb F_p}e_p(f(x))\right|\leq p
\leq(m-1)\sqrt p$.

To handle the condition that our variables belong to a specific interval, we use finite Fourier expansions. The following lemma estimates the Fourier coefficients of the indicator function of an interval.

\begin{lemma}[\cite{Bourgain}, Lemma 11.1]
\label{estimate_T}
    Let $H$ be any subset of $\mathbb{F}_p$ consisting of consecutive integers. Then its indicator function can be expanded as $\mathbf{1}_H(t) = \frac{1}{p} \sum_{a \in \mathbb{F}_p} T_H(a) e_p(at)$, where $T_H(a) = \sum_{y \in H} e_p(-ay)$. Furthermore, we have
    $$\sum_{a \neq 0} |T_H(a)| \leq \frac{4}{\pi^2} p \log (3p).$$
\end{lemma}

For quadratic polynomials, we rely on the following uniform bound for incomplete Gauss sums to improve the estimate.

\begin{lemma}[{\cite[Page 53, Theorem]{Ko} and \cite[Equation (1.5)]{Cochrane}}]
\label{quadratic_bound}
    Let $f(x) = Ax^2 + Bx + C \in \mathbb{Z}[x]$ be a quadratic polynomial. For any odd prime $p \nmid A$ and $1 \leq N \leq p$, the incomplete exponential sum satisfies
    \[ \left| \sum_{x=1}^{N} e_p(f(x)) \right| \leq 7.0508\sqrt{p} + 2 \ll \sqrt{p}. \]
\end{lemma}

Next, we list several known results about Dirichlet characters and $L$-functions. These will be used in Section \ref{original case} to deal with the monomial case $\varphi(x) = x^m$.
\begin{lemma}[\cite{MV}, Theorem 2]
\label{dirichlet character}
    Let $p$ be a prime number. Assuming the Generalized Riemann Hypothesis, for any non-principal character $\chi$ modulo $p$ and any $x$, we have 
    $$\sum_{n\leq x}\chi(n)\ll \sqrt{p}\log\log p.$$
\end{lemma}

To explicitly calculate the quadratic case $m=2$, we need the following basic properties of quadratic residues and the class number formula.

\begin{lemma}[\cite{CF}, Page 46, 4.5(c)(d) \& 4.6(a)]
\label{CF_lemma}
Let $p$ be an odd prime. Then 
\begin{enumerate}
    \item[(1)] For $p\equiv1\pmod{4}$, the number of quadratic residues in the interval $(0,\frac{p}{2})$ equals the number of nonresidues;
    \item[(2)] For $p\equiv3\pmod{4}$, there are more quadratic residues than the number of quadratic non-residues in the interval $(0,\frac{p}{2})$;
    \item[(3)] For $p\equiv3\pmod{4}$ and $p>3$, let $h(-p)$ denote the class number of $\mathbb{Q}(\sqrt{-p})$. Then we have 
    $$h(-p)=\frac{1}{2-\legendre{2}{p}}\sum_{n=1}^{\frac{p-1}{2}}\legendre{n}{p}.$$
\end{enumerate}
\end{lemma}

Finally, to prove that our error term is essentially the best possible, we need the following lower bound for the extreme values of Dirichlet $L$-functions.
\begin{lemma}[\cite{BCE}, Theorem 1(C)]
\label{L-function}
    There exists an absolute positive constant $c$ such that the inequality
    $$L(1, \chi_p) > c \log\log p$$
    holds for infinitely many primes $p \equiv 3 \pmod 4$, where $\chi_p(n) = \legendre{n}{p}$ is the Legendre symbol modulo $p$.
\end{lemma}
\section{Estimates for General, Quadratic, and Symmetric Polynomials}
\label{proof of main results}
In this section, we apply the analytic tools established in Section \ref{preliminaries} to provide asymptotic estimates for the counting function $f_\varphi(p)$. Our strategy involves translating the counting problem into an exponential sum over specific intervals, which we then handle using the finite Fourier expansion of indicator functions.

Let $I$ and $J$ be the sets of integers defined by
    $$I := \left\{1, 2, \dots, \frac{p-1}{2}\right\}, \quad J := \left\{\frac{p+1}{2}, \dots, p-1\right\}.$$
    Then the counting function can be rewritten
    $$f_\varphi(p) = \#\{x \in I : \varphi(x) \bmod p \in J\}.$$
    Recall that $e_p(t) = e^{2\pi it/p}$. The indicator function of the set $J$ can be expressed as finite Fourier series
    \[
    \mathbf{1}_J(t)  = \sum_{y\in J} \left(\frac{1}{p} \sum_{a \in \mathbb{F}_p} e_p(a(t-y))\right)  = \frac{1}{p} \sum_{a \in \mathbb{F}_p} e_p(at) T_J(a).
\]
    Substituting this into the counting function, we have
    \[
\begin{aligned}
    f_\varphi(p)  = \sum_{x \in I} \mathbf{1}_J(\varphi(x)) & = \sum_{x \in I}\left(\frac{1}{p} \sum_{a \in \mathbb{F}_p} T(a)  e_p(a\varphi(x))\right) \\
         & = \frac{1}{p} \sum_{a \in \mathbb{F}_p} T(a) \sum_{x \in I} e_p(a\varphi(x)),
\end{aligned}
\]
where $T(a):=T_J(a)$ for convenience.
 
This identity transforms our counting problem into the estimation of incomplete exponential sums over the domain $I$. For general polynomials, we proceed by applying a second finite Fourier expansion to the domain $I$, which leads to a double exponential sum.

\begin{proof}[Proof of Theorem \ref{general}]
    We express the indicator function of the set $I$ as
    $$\mathbf{1}_I(x) = \frac{1}{p} \sum_{b \in \mathbb{F}_p} U(b) e_p(bx), \quad \text{where } U(b) = \sum_{z \in I} e_p(-bz).$$
    Then the counting function can be written as 
\[
f_\varphi(p) = \sum_{x=0}^{p-1} \mathbf{1}_I(x) \mathbf{1}_J(\varphi(x)) = \frac{1}{p^2} \sum_{a \in \mathbb{F}_p} \sum_{b \in \mathbb{F}_p} T(a) U(b) \sum_{x=0}^{p-1} e_p(a\varphi(x) + bx),
\]
    Let $W(a,b) := \sum_{x=0}^{p-1} e_p(a\varphi(x) + bx)$. Since $W(0,0)=p$ and $W(0,b)= \sum_{x=0}^{p-1} e_p(bx) =0$ for $b\neq 0$, we have 
    \[
    \begin{aligned}
    f_\varphi(p) &= \frac{1}{p}T(0)U(0)+\frac{1}{p^2} \sum_{a \neq 0} \sum_{b \in \mathbb{F}_p} T(a) U(b) W(a,b)\\
    &=\frac{(p-1)^2}{4p}+\frac{1}{p^2} \sum_{a \neq 0} \sum_{b \in \mathbb{F}_p} T(a) U(b) W(a,b).
    \end{aligned}
    \]
    
    According to Lemma \ref{Weil bound}, since $p$ is an odd prime and $\deg(a\bar{\varphi}+bx)=\deg(\bar{\varphi})\geq 2$ for $a\neq0$, we have $$|W(a,b)|\leq(\deg(\bar{\varphi})-1)\sqrt{p}\leq(m-1)\sqrt{p}.$$
    By Lemma \ref{estimate_T}, we have $\sum_{a \neq 0} |T(a)| \leq \frac{4}{\pi^2}p\log(3p)$. Similarly, for the set $I$ we have 
    $$\sum_{b \in \mathbb{F}_p} |U(b)| = |U(0)| + \sum_{b \neq 0} |U(b)| \leq \frac{p-1}{2} + \frac{4}{\pi^2}p\log(3p)$$ 
    Consequently, 
    \[
    \begin{aligned}
    \left|f_\varphi(p)-\frac{(p-1)^2}{4p}\right|&=\left|\frac{1}{p^2} \sum_{a \neq 0} \sum_{b \in \mathbb{F}_p} T(a) U(b) W(a,b)\right|\\ &\leq \frac{1}{p^2} \sum_{a \neq 0} \sum_{b \in \mathbb{F}_p} |T(a)| |U(b)| |W(a,b)| \\
    &\leq \frac{(m-1)\sqrt{p}}{p^2}\left( \sum_{a \neq 0} |T(a)| \right) \left( \sum_{b \in \mathbb{F}_p} |U(b)| \right)\\
    &\leq \frac{4(m-1)}{\pi^2}\sqrt{p}\log(3p)\left(\frac{1}{2}+\frac{4}{\pi^2}\log(3p)\right).
    \end{aligned}\]
    This proves the explicit estimate in the theorem. Since $\varphi$ is a fixed polynomial with $\deg(\varphi)\geq2$, there are only finitely many primes $p$ for which the reduction of $\varphi$ modulo $p$ has degree less than 2. For each of them, we use the trivial bound as stated in Remark \ref{remark1.2}. 
    Therefore, we have 
$$f_\varphi(p) = \frac{(p-1)^2}{4p} + O_\varphi(\sqrt{p}\log^2 p).$$
Noting that $\frac{(p-1)^2}{4p} = \frac{p}{4} - \frac{1}{2} + \frac{1}{4p}$ and the bounded terms $-\frac{1}{2} + \frac{1}{4p} = O(1)$ can be absorbed into the $O_\varphi(\sqrt{p}\log^2 p)$ error term, we have $f_\varphi(p) = \frac{p}{4} + O_\varphi(\sqrt{p}\log^2 p)$. 
\end{proof}

The $O_\varphi(\sqrt{p}\log^2 p)$ error term in Theorem \ref{general} arises from applying finite Fourier expansions to both the domain $I$ and the range $J$. For quadratic polynomials, we can avoid the Fourier expansion over the domain $I$ by directly estimating the incomplete exponential sum. This allows us to remove one logarithmic factor in the error term.

\begin{proof}[Proof of Theorem \ref{quadratic}]
    For all primes $p \nmid A$ and for every $a \neq 0$, the sum
    \[
        S(a):=\sum_{x=1}^{(p-1)/2}e_p(a\varphi(x))
    \]
    is an incomplete quadratic Gauss sum. According to the classical uniform bounds for incomplete Gauss sums stated in Lemma \ref{quadratic_bound}, we have
    \[
        |S(a)| \le 7.0508\sqrt{p} + 2 
    \]
    uniformly for all $a \neq 0$. 

    Using the one-dimensional Fourier expansion for the indicator function of $J$ as before, we have
    \[
        f_\varphi(p) = \frac{|I||J|}{p} + \frac{1}{p}\sum_{a\neq 0}T_J(a)S(a).
    \]
    Therefore, by Lemma \ref{estimate_T}, we have
    \[
\left|
f_\varphi(p)-\frac{(p-1)^2}{4p}
\right|
\leq
\frac1p
\left(\max_{a\ne0}|S(a)|\right)
\sum_{a\ne0}|T_J(a)|\leq
\frac4{\pi^2}
\left(7.0508\sqrt p+2\right)\log(3p).
\]
    For finitely many primes $p$ with $p\,|\,A$, we use the trivial bound as stated in Remark \ref{remark1.2}, and thus have $f_\varphi(p)=\frac{p}{4}+O_\varphi(\sqrt{p}\log(p))$.
\end{proof}

For quadratic polynomials, we have proved the estimate $f_\varphi(p)=\frac{p}{4}+O_\varphi(\sqrt{p}\log p)$, which removes one logarithmic factor by specific bounds for incomplete Gauss sums. For polynomials of higher degrees, a similar refinement is possible if $\varphi(x)$ possesses reflection symmetry. If the polynomial satisfies the symmetry condition $\varphi(x) \equiv \varphi(c-x) \pmod p$, we can relate the incomplete sum over the half-interval $I$ directly to a complete exponential sum over the whole field $\mathbb{F}_p$. This alternative approach successfully avoids the Fourier expansion over the domain $I$ and yields an improved estimate of $O_\varphi(\sqrt{p} \log p)$.

\begin{proof}[Proof of Theorem \ref{general_symmetry}]
    For $a \in \mathbb{F}_p$, consider the sum 
    $$S(a) = \sum_{x \in I} e_p(a\varphi(x)).$$
    We introduce the involution $\tau(x) = c - x \pmod p$. Since $c$ is a fixed integer, the symmetry $\varphi(\tau(x)) \equiv \varphi(x) \pmod p$ allows us to rewrite $S(a)$ by substituting $y = \tau(x)$ as follows
    $$S(a) = \sum_{x \in I} e_p(a\varphi(c-x)) = \sum_{y \in \tau(I)} e_p(a\varphi(y)).$$
    By the inclusion-exclusion principle, we have the identity
    \begin{equation}
        2S(a) = \sum_{x \in I \cup \tau(I)} e_p(a\varphi(x)) + \sum_{x \in I \cap \tau(I)} e_p(a\varphi(x)). \label{eq:inclusion_exclusion}
    \end{equation}
    
    We first bound the size of the overlap $I \cap \tau(I)$. An element $x \in I \cap \tau(I)$ must satisfy $1 \le x \le (p-1)/2$ and $1 \le c-x \pmod p \le (p-1)/2$. Let $y = c-x \pmod p$, then $x+y \equiv c \pmod p$. Since $x, y \in I$, their sum satisfies $2 \le x+y \le p-1$. For a fixed $c$ and $p\geq2|c|+3$, the congruence $x+y \equiv c \pmod p$ implies the exact equality $x+y = c$ if $c\geq 2$ or $x+y = c+p$ if $c\leq -1$. For $c=0$ or 1, there are no solutions. For $p<2|c|+3$, $|I\cap\tau(I)|\leq|I|=\frac{p-1}{2}\leq|c|+1$. In all cases, the number of such integer pairs $(x, y)$ is at most $|c|+1$, which is a constant independent of $p$. 
    
    Similarly, the union $I \cup \tau(I)$ covers $\mathbb{F}_p$ except for a set $C$ of constant size, where $|C| = p - |I \cup \tau(I)| = 1 + |I \cap \tau(I)| \leq |c|+2$. 
    Therefore, the first sum in \eqref{eq:inclusion_exclusion} can be completed to a sum over $\mathbb{F}_p$ as
$$\sum_{x \in I \cup \tau(I)} e_p(a\varphi(x)) = \sum_{x \in \mathbb{F}_p} e_p(a\varphi(x)) - \sum_{x \in C} e_p(a\varphi(x)).$$
Substituting this back into \eqref{eq:inclusion_exclusion} and taking modulus, we have
\[
\begin{aligned}
2|S(a)| &= \left|\sum_{x \in \mathbb{F}_p} e_p(a\varphi(x))-\sum_{x \in C} e_p(a\varphi(x))+\sum_{x \in I \cap \tau(I)} e_p(a\varphi(x))\right|\\
&\leq\left|\sum_{x \in \mathbb{F}_p} e_p(a\varphi(x))\right|+\left|\sum_{x \in C} e_p(a\varphi(x))\right|+\left|\sum_{x \in I \cap \tau(I)} e_p(a\varphi(x))\right|\\
&\leq\left|\sum_{x \in \mathbb{F}_p}e_p(a\varphi(x))\right|+2|c|+3.
\end{aligned}
\]
According to Lemma \ref{Weil bound}, if $\bar{\varphi}$ is nonconstant, then, for every $a\neq0$, the complete sum is bounded by $(m-1)\sqrt{p}$. Combining these results, we obtain the bound
$$\max_{a \neq 0} |S(a)| \leq \frac{(m-1)\sqrt{p}+2|c|+3}{2}.$$
On the other hand, according to Lemma \ref{estimate_T}, we have 
$$\sum_{a \neq 0} |T(a)|\leq\frac{4}{\pi^2}p\log(3p).$$
Hence, by the triangle inequality,
\[
\begin{aligned}
\left|f_\varphi(p)-\frac{(p-1)^2}{4p}\right|&=\left|\frac{1}{p} \sum_{a \neq 0} S(a) T(a)\right|\\
&\leq \frac{1}{p} \left( \max_{a \neq 0} |S(a)| \right) \sum_{a \neq 0} |T(a)|\\
 &\leq \frac{2}{\pi^2}\left((m-1)\sqrt{p}+2|c|+3\right)\log(3p).
\end{aligned}
\]
Since the reduction $\bar{\varphi}$ is constant for only finitely many primes $p$, we have
$$f_\varphi(p) = \frac{(p-1)^2}{4p} + O_\varphi(\sqrt{p}\log p)= \frac{p}{4} + O_\varphi(\sqrt{p}\log p).$$
This completes the proof.
\end{proof}

\section{Exact Formula and Conditional Bounds  for the Monomial Case}
\label{original case}
In the previous section, we established asymptotic bounds for general, quadratic, and symmetric polynomials. In this section, we specialize our study to the monomial case $\varphi(x) = x^m$ for an even integer $m \ge 2$. We begin by focusing on the quadratic case $m=2$. This specific case provides a bridge to the class numbers of imaginary quadratic fields, allowing us to explicitly calculate the counting function. This exact unconditional evaluation yields a sharp lower bound that disproves Sun's conjecture and sets the stage for our conditional estimates for higher degrees.

\begin{proof}[Proof of Theorem \ref{m=2}]
    As is proved in Section \ref{proof of main results}, the counting function is
    $$f_2(p)=\#\{k\in I: k^2\bmod{p}\in J\},$$
    which counts the quadratic residues of elements in $J$ belonging to $I$. Hence using the indicator function and properties of the Legendre symbol, the counting function can be written as 
    \[
        f_2(p)=\sum_{n\in J}\frac{1+\legendre{n}{p}}{2}
        =\frac{|J|}{2}+\frac{1}{2}\sum_{n=\frac{p+1}{2}}^{p-1}\legendre{n}{p}
        =\frac{p-1}{4}+\frac{1}{2}\sum_{n=1}^{\frac{p-1}{2}}\legendre{-n}{p}
        =\frac{p-1}{4}+\frac{1}{2}\legendre{-1}{p}\sum_{n=1}^{\frac{p-1}{2}}\legendre{n}{p}.
    \] 
    By Lemma \ref{CF_lemma} (1)(2), in this interval, the number of the quadratic residues is equal to that of quadratic non-residues for $p\equiv 1\pmod{4}$. So we have 
    \[
f_2(p)=
\begin{cases}
    \frac{p-1}{4} & \text{if } p\equiv 1\pmod{4}\\
    \frac{p-1}{4}-\frac{1}{2}\sum_{n=1}^{(p-1)/2}\legendre{n}{p} & \text{if } p\equiv 3\pmod{4} 
\end{cases}.
\]
Let $h(-p)$ denote the class number of $K=\mathbb{Q}(\sqrt{-p})$. According to Lemma \ref{CF_lemma} (3), for $p>3$ and $p\equiv 3\pmod{4}$, the sum of Legendre symbols can be written as 
$$\frac{1}{2}\sum_{n=1}^{(p-1)/2}\legendre{n}{p}=\frac{1}{2}h(-p)\cdot \left(2-\legendre{2}{p}\right).$$
Hence we have the explicit formula for $p>3$
    \[
f_2(p)=
\begin{cases}
    \frac{p-1}{4} & \text{if } p\equiv 1\pmod{4}\\
    \frac{p-1}{4}-\frac{3}{2}h(-p) & \text{if } p\equiv 3\pmod{8}\\
    \frac{p-1}{4}-\frac{1}{2}h(-p) & \text{if } p\equiv 7\pmod{8}
\end{cases}.
\]

    From Dirichlet's analytic class number formula for imaginary quadratic fields, we have $h(-p) = \frac{\sqrt{p}}{\pi} L(1, \chi_p)$. According to the unconditional $\Omega$-result in Lemma \ref{L-function}, there exists an absolute constant $c>0$ such that $L(1, \chi_p) > c \log\log p$ for infinitely many primes $p\equiv 3\pmod{4}$. Therefore, for these primes, 
    $$ \left|f_2(p)-\frac{p}{4}\right| \geq \frac{c}{2\pi}\sqrt{p}\log\log p. $$ 
    Taking $C=\frac{c}{2\pi}$ completes the proof.
\end{proof}

For $m=2$, Littlewood \cite[Theorem 1]{Littlewood} proved that $L(1,\chi_p)=O(\log\log p)$ under the Generalized Riemann Hypothesis, which implies the conditional upper bound $f_2(p) = p/4 + O(\sqrt{p}\log\log p)$. While the exact formula is restricted to $m=2$, we can investigate higher even powers $m \ge 2$ by noticing that the monomial $x^m$ has a rich multiplicative structure over finite fields. This observation allows us to translate the counting problem into the estimation of incomplete Dirichlet character sums. Assuming the Generalized Riemann Hypothesis for Dirichlet \(L\)-functions, we can generalize this upper bound to all even integers $m$.

\begin{proof}[Proof of Theorem \ref{estimate_GRH}]
    Since the monomial $x^m$ is an even function for an even positive integer $m$, the sum over the half-interval $I$ can be directly extended to the whole group $\mathbb{F}_p^*$. Thus, the counting function can be written as $f_m(p) = \frac{1}{2}\sum_{k=1}^{p-1} \mathbf{1}_J(k^m)$. Let $N_m(y)$ denote the number of solutions to $x^m \equiv y \pmod p$. Using the multiplicative character expansion, $N_m(y) = \sum_{\chi^m = \chi_0} \chi(y)$. Substituting this into the counting function and exchanging the order of summation, we have 
    $$f_m(p) = \frac{1}{2} \sum_{\chi^m = \chi_0} \sum_{y \in J} \chi(y) = \frac{p-1}{4} + \frac{1}{2} \sum_{\substack{\chi^m = \chi_0 \\ \chi \neq \chi_0}} \sum_{y \in J} \chi(y).$$
    The error term is a finite sum of incomplete Dirichlet character sums over the interval $J$. According to Lemma \ref{dirichlet character}, under the Generalized Riemann Hypothesis, each inner sum can be bounded by $O(\sqrt{p}\log\log p)$. Since the number of non-principal characters in the outer sum is bounded by $m-1$, the total error term is bounded by $O_m(\sqrt{p}\log\log p)$.  
\end{proof}

\begin{remark}
       Theorem \ref{m=2} shows that the factor \(\log\log p\) is unavoidable
at least in the quadratic case.  Thus Sun's conjectural
\(O(\sqrt p)\) error term cannot hold uniformly for all \(m\).
Theorem \ref{estimate_GRH} gives, under GRH, a matching upper-order
estimate for all even degree monomials.
\end{remark}
\section{Generalization of the Counting Function}
\label{section_generalizaion}
In this section, we show that the $O(\sqrt{p}\log p)$ bound holds for polynomials invariant under the $G_d$-action by evaluating the sum over a fundamental domain.
\begin{proof}[Proof of Theorem \ref{fundamental domain}]
Using the same method in the proof of Theorem \ref{general}, we write
$$f_A(p) = \frac{1}{p} \sum_{a \in \mathbb{F}_p} T(a) S_A(a),$$
where $T(a) = \sum_{y \in J} e_p(-ay)$ and $S_A(a) = \sum_{x \in A} e_p(ax^d)$. For $a=0$, the term is $\frac{1}{p} |J| |A| = \frac{1}{p} \left(\frac{p-1}{2}\right) \left(\frac{p-1}{d}\right) = \frac{(p-1)^2}{2dp}$. When $a \neq 0$, we use the invariance under the $G_d$-action. For any integer $0 \leq i \leq d-1$ and any $x \in A$, we have $(\omega^i x)^d = \omega^{id} x^d \equiv x^d \pmod p$. Therefore, the incomplete sum over any orbit remains identical
$$\sum_{x \in \omega^i A} e_p(ax^d) = \sum_{x \in A} e_p(ax^d) = S_A(a).$$
We decompose the complete exponential sum $W(a) := \sum_{x=0}^{p-1} e_p(ax^d)$ over the disjoint orbits and get
$$W(a) = e_p(0) + \sum_{i=0}^{d-1} \sum_{x \in \omega^i A} e_p(ax^d) = 1 + d S_A(a),$$
which means $S_A(a) = \frac{W(a) - 1}{d}$. According to the Weil bound, $|W(a)| \le (d-1)\sqrt{p}$. Thus, we have
$$|S_A(a)| \le \frac{(d-1)\sqrt{p} + 1}{d}.$$
Applying Lemma \ref{estimate_T}, the error term is bounded by

\[
\begin{aligned}
   \left|f_A(p)-\frac{(p-1)^2}{2dp}\right|&=\left|\frac{1}{p}\sum_{a\neq 0}T(a)S_A(a)\right|\\&\leq \frac{1}{p} \left( \max_{a \neq 0} |S_A(a)| \right) \sum_{a \neq 0} |T(a)|\\
 &\leq \frac{4}{d\pi^2}\left((d-1)\sqrt{p}+1\right)\log(3p).
\end{aligned}
\]
Therefore, we have $f_A(p)=\frac{p}{2d}+O_d(\sqrt{p}\log p)$. 
\end{proof}

\section*{Acknowledgments}
The authors are very grateful to the reviewer for their constructive suggestions, which have significantly improved the quality of this paper. The authors are supported by National Natural Science Foundation of China (No. 12231009, No. 12471006).

\end{document}